\documentclass[11pt,twoside,reqno,psamsfonts]{amsart}

\usepackage[left=2.9cm,top=3cm,right=2.9cm]{geometry}              
\usepackage[colorlinks = true, citecolor = black, linkcolor = black, urlcolor = black, pdfstartview=FitH]{hyperref}  
\usepackage[pdftex]{graphicx}
\usepackage[utf8]{inputenc}
\usepackage{microtype, float}
\usepackage{amsmath, verbatim}
\usepackage{amssymb,mathtools}
\usepackage{epstopdf}
\usepackage{epsfig}
\usepackage{mathscinet}

\usepackage{marginfix}
\usepackage{marginnote}
\let\oldtocsection=\tocsection
\let\oldtocsubsection=\tocsubsection
\let\oldtocsubsubsection=\tocsubsubsection
\renewcommand{\tocsection}[2]{\hspace{0em}\textbf{\oldtocsection{#1}{#2}}}
\renewcommand{\tocsubsection}[2]{\hspace{1.8em}\oldtocsubsection{#1}{#2}}
\renewcommand{\tocsubsubsection}[2]{\hspace{3em}\oldtocsubsubsection{#1}{#2}}

\DeclareUnicodeCharacter{00A0}{ } 

\numberwithin{equation}{section}

\theoremstyle{plain}

\newtheorem{theorem}{Theorem}[section]
\newtheorem{lemma}[theorem]{Lemma}
\newtheorem {corollary}[theorem]{Corollary}
\newtheorem{proposition}[theorem]{Proposition}

\newtheorem{claim}[theorem]{Claim}

\theoremstyle{definition}

\newtheorem {definition}[theorem]{Definition}

\newtheorem {remark}[theorem]{Remark}

\theoremstyle{remark}

\newcommand{\R}{\mathbb{R}}

\newcommand{\C}{\mathbb{C}}
\newcommand{\Q}{\mathbb{Q}}

\newcommand{\I}{\mathbb{I}}

\renewcommand{\emptyset}{\varnothing}
\renewcommand{\epsilon}{\varepsilon}
\renewcommand{\rho}{\varrho}
\renewcommand{\phi}{\varphi}

\renewcommand{\hat}{\widehat}

\renewcommand{\Im}{\mathrm{Im\,}}
\renewcommand{\Re}{\mathrm{Re\,}}

\newcommand{\trdeg}{\text{tr.deg}}

\newcommand{\F}{\mathcal{F}}

\newcommand{\Rbar}{\overline{\mathbb{R}}}

\newcommand{\V}{\mathcal{V}}
\newcommand{\Rexp}{\mathbb{R}_{\exp}}

\newcommand{\Hbb}{\mathbb{H}}
\newcommand{\Zbb}{\mathbb{Z}}

\newcommand{\Vcal}{\mathcal{V}}

\newcommand{\Omegabar}{\overline{\Omega}}

\newcommand{\wptilde}{\tilde{\wp}}
\newcommand{\xtilde}{\tilde{x}}
\newcommand{\ytilde}{\tilde{y}}

\newcommand{\Wcal}{\mathcal{W}}

\newcommand{\Fhat}{\hat{F}}
\newcommand{\ybar}{\bar{y}}

\newcommand{\zbar}{\bar{z}}

   \def\XXint#1#2#3{{\setbox0=\hbox{$#1{#2#3}{\int}$}
        \vcenter{\hbox{$#2#3$}}\kern-.5\wd0}}

\title[Nondefinability results for elliptic and modular functions]{Nondefinability results for elliptic and modular functions}

\author{McCulloch, Raymond}
\address{}
\email{raymond.mcculloch@manchester.ac.uk}


 \thanks{\noindent Affiliation: University of Manchester. ORCID ID: 0000-0002-0570-4977\\This work formed part of the author's PhD thesis, which was supported by an Engineering and Physical Sciences Research Council Doctoral Training Award.\\The author is also is grateful to the Heilbronn Institute for Mathematical Research for support.}

\begin{document}
	
	\begin{abstract}
Let $\Omega$ be a complex lattice which does not have complex multiplication and $\wp=\wp_\Omega$ the Weierstrass $\wp$-function associated to it. Let $D\subseteq\mathbb{C}$ be a disc and $I\subseteq\mathbb{R}$ be a bounded closed interval such that $I\cap\Omega=\emptyset$. Let $f:D\rightarrow\mathbb{C}$ be a function definable in $(\overline{\mathbb{R}},\wp|_I)$. We show that if $f$ is holomorphic on $D$ then $f$ is definable in $\overline{\mathbb{R}}$. The proof of this result is an adaptation of the proof of Bianconi for the $\mathbb{R}_{\exp}$ case. We also give a characterization of lattices with complex multiplication in terms of definability and a nondefinability result for the modular $j$-function using similar methods.
	\end{abstract}

\maketitle
\section{Introduction}

	Model theorists have for some time been interested in definability questions concerning structures given by expanding the ordered real field $\Rbar$ by certain functions. For example the sine function is not definable in $\Rexp$, an immediate consequence of the o-minimality of $\Rexp$, which is proved by combining a result of Wilkie in \cite{wilkieexp} and work of Khovanski in \cite{khovanski}. Here and throughout this paper definable means definable with parameters in $\R $. In \cite{bianconi} Bianconi went further and showed that no non-trivial restriction of sine to a real interval is definable in $\Rexp$. This result may be rephrased to say that no restriction of the exponential function to an open disc $D$ in $\C$ is definable in $\Rexp$. Extending this further Bianconi showed in \cite{bianconiundef} that if $f:D\rightarrow\C$ is holomorphic and definable in $\Rexp$ then $f$ is algebraic. In \cite{newtonpetstar} Peterzil and Starchenko use this result to characterise all definable locally analytic subsets of $\C^n$ in $\Rexp$.

\par

This question of definability can in fact be generalised to other transcendental functions. Indeed such an example occurs with a transcendental function similar to the exponential function. Consider a \textit{complex lattice} $\Omega\subseteq\C$, a discrete subgroup of rank 2. Associated to each such lattice is the function

 $$\wp(z)=\wp_\Omega(z)=\frac{1}{z^2}+\sum_{\substack{\omega\in\Omega\setminus\{0\}}}\left(\frac{1}{(z-\omega)^2}-\frac{1}{\omega^2}\right).$$
 This function is similar to the exponential function as they are both periodic and have an addition formula as well as a differential equation. Also over the complex field an \textit{elliptic curve} $E(\C)=E_\Omega(\C)\subseteq\mathbb{P}(\C)$ is given by the equation $Y^2Z=4X^3-g_2XZ^2-g_3Z^3$, where the complex numbers $g_2$ and $g_3$ depend on the lattice $\Omega$ and are known as the \textit{invariants} of $\wp_\Omega$. The map $\exp_E:\C\rightarrow E(\C),z\mapsto[\wp(z):\wp'(z):1]$ is called the \textit{exponential map} of $E$. These similarities and the well known model theory of the exponential function make the model theory of the Weierstrass $\wp$-function a natural thing to consider. This has been done by various authors including Bianconi in \cite{bianconimc}, Macintyre in \cite{macwp} as well as Peterzil and Starchenko in \cite{ps-uniformdefforP} and Jones, Kirby and Servi in \cite{jks}.
\par 

During his investigations into the model theory of these Weierstrass $\wp$-functions, Macintyre observed the following. If the lattice $\Omega=\mathbb{Z}+i\mathbb{Z}$ then the restriction of $\wp$ to any complex disc $D$ on which $\wp$ is analytic is definable in the structure $(\Rbar,\wp|_{[1/8,3/8]})$. The interval $[1/8,3/8]$ is chosen for convenience as it avoids both the poles of $\wp$ and the zeroes of $\wp'$. Any such interval may be chosen.
\par 
For the lattice $\Zbb+i\Zbb$ it can immediately be seen that $\wp(iz)=-\wp(z)$ and this is all that is required to prove Macintyre's observation. In particular there is a non integer complex number $\alpha$ such that $\alpha\Omega\subseteq\Omega$. A lattice with this property is said to have \textit{complex multiplication}. A complex lattice $\Omega$ is called a \textit{real lattice} if $\overline{\Omega}=\Omega$. The lattice $\Omega=\mathbb{Z}+i\mathbb{Z}$ is an example of a real lattice which has complex multiplication. In the preprint \cite{mcculloch2020nondefinability} Macintyre's result is extended to all real lattices with complex multiplication. It is also shown that if the restriction of $\wp$ to some open disc $D\subseteq\C$ is definable in the structure $(\Rbar,\wp|_I)$, where $I\subseteq\R$ is a closed interval this does not contain any lattice points and the lattice $\Omega$ is real, then the lattice $\Omega$ has complex multiplication. 
A direct extension of this result to semiabelian varieties is presumably false. For example consider the semiabelian variety $G=E\times \mathbb{G}_m$ where $E$ is an elliptic curve with complex multiplication and $\mathbb{G}_m$ is the multiplicative group. Then a restriction of $\exp_G$ to the real part of its fundamental domain will give the exponential map $\exp_E$ but will not give us, presumably, the full real exponential function.  
\par  Now we turn to extending the final aforementioned result of Bianconi to the $\wp$-function. The following theorem can be seen as a $\wp$-function analogue of Theorem 4 in \cite{bianconiundef}.
\begin{theorem}\label{theorem- undef for P theorem}
	Let $D\subseteq \R^{2n}$ be a definable open polydisc and $u,v:D\rightarrow\R$ be two functions that are both definable in the structure $(\Rbar,\wp|_I)$, where $\Omega$ is a complex lattice which does not have complex multiplication and $I$ is some bounded closed interval in $\R$ which does not contain a lattice point. Let $f(x,y)=u(x,y)+iv(x,y)$ be holomorphic in $D$. Then $u$ and $v$ are definable in $\Rbar$.
\end{theorem}
The proof of this theorem is given in Section \ref{section-proof of theorem} and adapts the method of Bianconi used to prove Theorem 4 in \cite{bianconiundef}. However the final part of the proof differs from Bianconi's argument as some of the conclusions are unclear. Bianconi's method involves using a theorem of Wilkie on smooth functions that are defined implicitly that was proved in general by Jones and Wilkie in \cite{jw}. However here we use an implicit definition obtained from a model completeness result due to Gabrielov in \cite{gab}. Although the theorem of Gabrielov is well known, as far as we are aware this is the first application of this result in order to obtain an implicit definition of this kind. These implicit definitions are given in Section \ref{section-implicit definition}.  
\par 
In Section \ref{section-further results} we give some nondefinability results for various transcendental functions, beginning with an analogue of the aforementioned result of Peterzil and Starchenko in \cite{newtonpetstar} for the Weierstrass $\wp$-function. Then we give a characterisation of the definability of restrictions of $\wp$ to a disc $D\subseteq\C$ in terms of the associated lattice $\Omega$ having complex multiplication, one direction of which follows from Theorem \ref{theorem- undef for P theorem}. This extends the result in \cite{mcculloch2020nondefinability} to all complex lattices. To complete this section we give a nondefinability result for the modular $j$-function the proof of which adapts a similar method to the proof of Theorem \ref{theorem- undef for P theorem}. Finally in Section \ref{section- final remarks} we give some concluding remarks on what other transcendental functions can give rise to similar nondefinability statements and the obstacles that prevent one from proving a version of Theorem \ref{theorem- undef for P theorem} for such functions using the method of Section \ref{section-proof of theorem}.

\section{The Weierstrass $\wp$ and modular $j$ functions}\label{section- p function background}
In this section we give background on both the Weierstrass $\wp$-function and the modular $j$-function.

\begin{definition}\label{definition- period ratio}
	Let $\Omega\subseteq\C$. Then $\Omega$ is said to be a \textit{complex lattice} if there exist complex numbers $\omega_1$ and $\omega_2$ such that $\Omega=\{ m\omega_1+n\omega_2:m,n\in\Zbb,\Im(\omega_2/\omega_1)>0 \}$. The set $\{ \omega_1,\omega_2  \}$ is referred to as an oriented basis for the lattice $\Omega$. The quotient $\tau=\omega_2/\omega_1\in\Hbb$ is known as the \textit{period ratio} of $\Omega$. The lattice generated by 1 and $\tau$ is denoted $\Omega_{\tau}=\langle1,\tau\rangle$.
\end{definition}

The following theorem can be seen in Chapter 3 of \cite{kc}.
\begin{theorem}
	For all $z\in\C\setminus\Omega$ we have that,
	
	\begin{equation}\label{differential equation for P}
	(\wp'(z))^2=4\wp^3(z)-g_2\wp(z)-g_3.
	\end{equation}
\end{theorem}

Therefore the functions $\wp$ and $\wp'$ are algebraically dependent. Differentiating both sides of this differential equation gives that 

\begin{equation}\label{equation- P'' formula}
\wp''(z)=6\wp^2(z)-\frac{g_2}{2}.
\end{equation} 

In particular for any $n\ge2$ the derivative $\wp^{(n)}$ may be written as a polynomial with complex coefficients in $\wp$ and $\wp'.$ Another crucial property of $\wp$ is its addition formula. This can be seen in Theorem 6 in Chapter 3 of \cite{kc}.
\begin{theorem}
	For complex numbers $z$ and $w$ such that $z-w\notin\Omega$ we have that, 
	
	\begin{equation}\label{equation- addition formula for P}
	\wp(z+w)=\frac{1}{4}\left(\frac{\wp'(z)-\wp'(w)}{\wp(z)-\wp(w)}\right)^2-\wp(z)-\wp(w).
	\end{equation}
	
\end{theorem}

The function $\wp'$ also has an addition formula. However this is less well known and may be deduced from the identity

\begin{equation}\label{equation- P function matrix identity}
\begin{vmatrix}
\wp(z)&\wp'(z)&1\\\wp(w)&\wp'(w)&1\\\wp(z+w)&-\wp'(z+w)&1
\end{vmatrix}=0,
\end{equation}

\noindent
which can be seen in page 363 of \cite{copson}. From this identity we have for all complex numbers $z$ and $w$ such that $z-w\notin\Omega$,

\begin{equation}\label{equation- addition formula for P'}
\wp'(z+w)=\frac{\wp(w)\wp'(z)-\wp'(w)\wp(z)-\wp(z+w)(\wp'(z)-\wp'(w))}{\wp(z)-\wp(w)}.
\end{equation}

This next definition can be seen in Section 4 of Chapter 1 of \cite{silvermanadvanced}. 
\begin{definition}
	The modular $j$-function is the function $j:\Hbb\rightarrow\C$ defined by, 
	
	$$j(\tau)=1728\frac{g_2^3(\tau)}{g_2^3(\tau)-27g_3^2(\tau)},$$ where the complex numbers $g_2$ and $g_3$ are the invariants of the complex lattice $\Omega$ with period ratio $\tau$. 
\end{definition}
It turns out that the modular $j$-function may be written rather differently, namely it has a $q$-expansion with (positive) integer coefficients. This may be seen in Proposition 7.4 in Chapter 1 of \cite{silvermanadvanced} and the explicit coefficients are in Example 6.2.2 of Chapter 2 of \cite{silvermanadvanced}.

\begin{proposition}
	Let $q=e^{2\pi i z}$. Then,
	$$j(z)=q^{-1}+744+196884q+21493760q^2+\dots.$$
\end{proposition}

\begin{remark}\label{remark- restriction of j is real valued}
	From the $q$-expansion it is clear that the restriction of $j$ to $\Hbb\cap i\R$ is a real valued function.
\end{remark}

By Theorem 4.1 in \cite{silvermanadvanced} the $j$-function is a modular function of weight zero. That is, for all $z,w\in\C$ we have that $j(z)=j(w)$ if and only if there is some matrix $\gamma\in SL_2(\Zbb)$ such that 

$$w=\frac{az+b}{cz+d},\text{ where }\gamma=\begin{pmatrix}
a&b\\c&d
\end{pmatrix}.$$ 

If $\gamma$ is a matrix in $GL^+_2(\mathbb{Q})$, the group of $2\times2$ matrices with rational entries and positive determinant, then there is a unique positive integer $M$ such that $M\gamma\in GL_2(\mathbb{Z})$ and the entries of $M\gamma$ are relatively prime. By Proposition 23 in \cite{123modularforms-zagier} we have that for each positive integer $M$ there is a polynomial $\Phi_M\in\mathbb{Z}[X,Y]$ such that $\Phi_M(j(z),j(w))=0$ if and only if there is a matrix $\gamma\in GL^+_2(\mathbb{Q})$ such that $z=\gamma w$ and $\det(M\gamma)=M$. Finally we note as in \cite{axforj} that $j$ satisfies a nonlinear third order differential equation, namely 

\begin{equation}\label{differential equation for j}
\frac{j'''}{j'}-\frac{3}{2}\left(\frac{j''}{j'}\right)^2+\left(\frac{j^2-1968j+2654208}{2j^2(j-1728)^2}\right)(j')^2=0.
\end{equation}
To conclude this section we state the versions of the Ax-Schanuel theorem for the Weierstrass $\wp$-function and the modular $j$-function. For the $\wp$-function this is due to Brownawell and Kubota and can be seen in \cite{bkpaper}.
\begin{theorem}\label{theorem-ax-schanuel for P-analytic functions on disc- multiple P functions}
	Suppose $\Omega_1,\dots,\Omega_m$ are complex lattices each of which does not have complex multiplication. Let $\tau_1,\dots,\tau_m$ be their corresponding period ratios and $\wp_1,\dots,\wp_m$ be their corresponding $\wp$-functions. Suppose that for all $i,j=1,\dots,m$ and $i\ne j$ there do not exist integers $a,b,c,d$ with $ad-bc\ne0$ such that $$\tau_j=\frac{a\tau_i+b}{c\tau_i+d}.$$
	Let $z_1,\dots,z_n$ be analytic functions on a disc $D$ centred at $\alpha\in\C$ and suppose that $z_1-z_1(\alpha),\dots,z_n-z_n(\alpha)$ are linearly independent over $\Q$. Then we have that 
	
	$$\trdeg_\C\C[z_1,\dots,z_n,\wp_1(z_1),\dots,\wp_1(z_n),\dots,\wp_m(z_1),\dots,\wp_m(z_n)]\ge nm+1.$$
\end{theorem}
The version of the Ax-Schanuel theorem for $j$ is due to Pila and Tsimerman in \cite{axforj}.
\begin{theorem}\label{theorem-ax schanuel for j-analytic functions}
	Let $z_1,\dots,z_n$ be analytic functions defined on a disc $D\subseteq \C$, which take values in the upper half plane, such that $j(z_1),\dots,j(z_n)$ are non-constant. Suppose that $\Phi_M(j(z_i),j(z_j))\ne0$ for all positive integers $M$ and for all $i,j=1,\dots,n$ where $i\ne j$. Then,
	
	$$\trdeg_{\mathbb{C}}\C[z_1,\dots,z_n,j(z_1),\dots,j(z_n),j'(z_1),\dots,j'(z_n),j''(z_1),\dots,j''(z_n)]\ge 3n+1.$$
\end{theorem}

\section{Implicit definitions}\label{section-implicit definition}
The purpose of each of these implicit definitions is to give a low upper bound on the transcendence degree of a finitely generated extension of $\C$. Before giving the first of these implicit definitions we give a precise definition of a property used in the statement of these implicit definitions.
\begin{definition} 
	
	Let $\F$ be a countable collection of real analytic functions defined on a bounded interval $I$ in $\R$. Let $f\in \F$. If the derivatives of $f$ may be written as a polynomial with coefficients in $\C$ in terms of a finite number of the functions in $\F$ then we say that the set $\F$ is \textit{closed under differentiation}.
	
	\par 
	
	Consider the structure $(\Rbar,\F)$ with $\F$ as above. Then if all the derivatives of the functions defined by terms are also defined by terms we say that the structure $(\Rbar,\F)$ has a \textit{ring of terms that is closed under differentiation.} 
\end{definition}

\subsection{Desingularisation}\label{desingularisation section}\label{section-desingularisation theorem}
\par 
The first implicit definition comes from ideas of Wilkie in \cite{wilkieexp} and is referred to by Bianconi in \cite{bianconi} as the Desingularisation Theorem. A more general form of this implicit definition was proved by Jones and Wilkie in \cite{jw}. Let $\tilde\R=(\Rbar,\F)$ be an expansion of $\Rbar$ by a set $\F$ of total analytic functions in one variable, closed under differentiation. We also assume that $\tilde\R$ has a model complete theory and as $\F$ is closed under differentiation the ring of terms of $\tilde\R$ is closed under differentiation. Before stating the first implicit definition we give a definition.

\begin{definition}
	Let $f_1:I\rightarrow\R$, for some open interval $I\subseteq\R$, be a function definable in the structure $\tilde\R=(\Rbar,\F)$. Then we say that $f_1$ is \textit{implicitly $\F$-defined} if there are some integers $n,l\ge 1$, polynomials $P_1,\dots,P_n$ in $\R[y_1,\dots,y_{(l+1)(n+1)}]$ and functions $f_2,\dots,f_n:I\rightarrow\R$ such that for all $z\in I$, 
	
	$$\begin{array}{ccc}
	F_1(z,f_1(z),\dots,f_n(z))=0\\\vdots\\F_n(z,f_1(z),\dots,f_n(z))=0
	\end{array}$$
	\noindent
	and $$\det\left(\frac{\partial F_i}{\partial x_j}\right)_{\substack{i=1,\dots,n\\j=2,\dots,n+1}}(z,f_1(z),\dots,f_n(z))\ne 0,$$ where 
	
	\begin{align*}F_i(z,f_1(z),\dots,f_n(z))=P_i(&z,f_1(z),\dots,f_n(z),\\&g_1(z),g_1(f_1(z)),\dots,g_1(f_n(z)),
	\dots,\\&g_l(z),g_l(f_1(z)),\dots,g_l(f_n(z)))\end{align*} 
	for $g_1,\dots,g_l\in\F.$ 
\end{definition}

\begin{theorem}[Jones \& Wilkie]\label{desingularisation theorem}\label{theorem-desingularisation}
	Let $f:I\rightarrow\R$, for some open interval $I\subseteq\R$, be a definable function in $\tilde\R$. Then there are subintervals $I_1,\dots,I_m\subseteq I$ such that $I\setminus\left(\cup_{k=1}^m I_k\right)$ is a finite set and $f$ is implicitly $\F$-defined on each of these subintervals.

\end{theorem}

\subsection{An implicit definition following from a result of Gabrielov}
\par 
This implicit definition is obtained from a model completeness result of Gabrielov in \cite{gab}. As noted in the introduction, although the theorem of Gabrielov is well known, as far as I am aware this is the first application of this theorem in order to obtain an implicit definition of this kind. Firstly we state Gabrielov's theorem and give some background terminology from \cite{gab}. Then we state and prove the implicit definition.

\begin{theorem}[Gabrielov]\label{Gabrielov theorem}\label{theorem- Gabrielov }
	Let $Y$ be a $\Phi$-subanalytic subset of $[0,1]^n$. Then $\tilde{Y}=[0,1]^n\setminus Y$ is $\Phi$-subanalytic.	
\end{theorem}

Consider a set of restricted real analytic functions $\Phi$ and a subanalytic set $Y$ defined from the functions in $\Phi$. Then by the previous theorem the complement of $Y$ is defined by functions in the algebra generated by the functions in $\Phi$, their partial derivatives, the constants 0 and 1 and the coordinate functions. In particular we have the following corollary.

\begin{corollary}[Gabrielov]\label{corollary-background-gabrielov- model completeness}
	Let $\F$ be an infinite collection of real analytic functions that are defined on a bounded closed interval in $\R$ that is closed under differentiation. Then the structure $(\Rbar,\F)$ is model complete.
\end{corollary}

The following lemma is Lemma 3 in \cite{gab} and is required for the proof of the implicit definition.

\begin{lemma}\label{lemma- gabrielov-lemma 3}
	Let $X$ be a $\Phi$-semianalytic set in $[0,1]^{m+n}$, and let $Y=\pi X\subseteq [0,1]^n,d=\dim Y$. Then there exist finitely many $\Phi$-semianalytic subsets $X_v^\prime$ and a $\Phi$-subanalytic subset $V$ of $X$ such that $Y=\left(\pi V\right)\cup\bigcup_v \pi X_v^\prime$ and 
	
	\begin{enumerate}
		\item $X_v^\prime$ is effectively non-singular, $\dim X_v^\prime=d$ and $\pi:X_v^\prime\rightarrow Y$ has rank $d$ at every point of $X_v^\prime$ for each $v$.
		\item $\dim \pi V<d$ 
		\item $X_u^\prime\cap X_v^\prime=\emptyset,$ for $u\ne v$.
	\end{enumerate}
\end{lemma}

Now we shall state and prove the implicit definition that arises from Gabrielov's theorem.

\begin{proposition}\label{proposition- gab imp definition }
	Let $\F$ be a set of real analytic functions defined on a neighbourhood in $[0,1]$ that contains a closed interval $I$, suppose that $\F$ is closed under differentiation and consider the structure $(\Rbar,\F|_I)$, where $\F|_I\coloneqq\{ g|_I:g\in\F \}$. Let $f:U\rightarrow I^k$ where $U\subseteq I^m$ for some $m,k\ge1$ be a function definable in $(\Rbar,\F)$ and let $f_1,\dots,f_k:U\rightarrow I$ denote its coordinate functions.

	\par
	
	Then there exist integers $n,l \ge 1$, polynomials $P_1,\dots,P_n$ in $\R[y_1,\dots,y_{(l+1)(m+n)}]$, functions $f_{k+1},\dots,f_n:B\rightarrow I$ for an open box $B\subseteq U$ and $g_1,\dots,g_l\in \F$ such that for all $\zbar=(z_1,\dots,z_m)\in B$,
	
	$$\begin{array}{ccc}
	F_1(\zbar,f_1(\zbar),\dots,f_n(\zbar))=0\\\vdots\\F_{n}(\zbar,f_1(\zbar),\dots,f_n(\zbar))=0
	\end{array}$$ and $$\det\left(\frac{\partial F_i}{\partial x_j}\right)_{\substack{i=1,\dots,n\\j=m+1,\dots,m+n}}(\zbar,f_1(\zbar),\dots,f_n(\zbar))\ne0,$$
	\noindent
	where
	
	\begin{align*}F_i(\zbar,f_1(\zbar),\dots,f_n(\zbar))=P_i(&\zbar,f_1(\zbar),\dots,f_n(\zbar),\\&g_1(z_1),\dots,g_1(z_m),g_1(f_1(\zbar)),\dots,g_1(f_n(\zbar)),
	\dots,\\& g_l(z_1),\dots,g_l(z_m),g_l(f_1(\zbar)),\dots,g_l(f_n(\zbar))).
	\end{align*} 
\end{proposition}

\begin{proof}
	
	Here the functions in $\F$ are defined on a neighbourhood in $[0,1]$ rather than a neighbourhood containing $[0,1]$. This has a slight impact on the definitions and results of Gabrielov that we wish to apply, namely that the interval $I\subseteq[0,1]$ takes the place of $[0,1]$ in the above statements. Let $Y=\Gamma(f)\subseteq \R^{m+1}$ be the graph of $f$. Clearly $\dim Y=m$. Then $Y$ is a definable set in the structure $(\Rbar,\F)$ and by Corollary \ref{corollary-background-gabrielov- model completeness} the set $Y$ is a $\F$-subanalytic set of dimension $m$. By definition $Y=\pi X$ where $X$ is a $\F$-semianalytic subset of $\R^{m+n}$ for some $n$. By Lemma \ref{lemma- gabrielov-lemma 3} we have that $Y=\left(\pi V\right)\cup\bigcup \pi X_v'$ where $X_v'$ are effectively non-singular $\F$-semianalytic sets of dimension $m$ and $\pi V$ is small. It is enough to prove the result for $Y=\pi X_v'$ for a single effectively non-singular set $X_v'$. By the definition of an effectively non-singular set and the rank condition seen in Definition 3 in \cite{gab} the function $f$ may be defined by a non-singular system of $m+n-m$ equations as described in the statement.
\end{proof}

\section{Proof of Theorem \ref{theorem- undef for P theorem}}\label{section-proof of theorem}
The proof of Theorem \ref{theorem- undef for P theorem} consists of three cases. Namely, when the lattice $\Omega$ is closed under complex conjugation (a real lattice), when it is isogenous to its conjugate and when it is not. The method for each of these cases is essentially the same and here we give the proof in the case when $\Omega$ is a real lattice. The differences between the proof of the real lattice case and the other two cases are explained at the end of this section. 
\par 
Assume that $\Omega$ is a real lattice. Then the restriction $\wp|_I$ is a real valued function, this can be seen in Section 18 of \cite{duval}. From the differential equation it is clear that the structures $(\Rbar,\wp|_I)$ and $(\Rbar,\wp|_I,\wp'|_I)$ are the same in the sense of having the same definable sets and it therefore suffices to prove the theorem using the structure $(\Rbar,\wp|_I,\wp'|_I)$. By Gabrielov's theorem, Theorem \ref{Gabrielov theorem}, this structure is model complete. Model completeness results involving the $\wp$-function are also due to Bianconi in \cite{bianconimc}. However these results deal with complex functions rather than their restrictions to a real interval and therefore do not seem applicable here.
\par 
If $n>1$ then we can fix all the variables except one and apply the $n=1$ case for each variable in turn. Therefore each coordinate function of $f$ is semialgebraic and holomorphic and so $f$ is an algebraic function in each variable and by Theorem 2 in \cite{sharipov-sukhov} the function $f$ is itself algebraic and therefore definable in $\Rbar$. Hence we may assume that $n=1.$
\par 
Assume for a contradiction that $v$ is not definable in $\Rbar$. The proof of the following claim is a straightforward application of the identities for the real and imaginary parts of a complex function and so we simply state this claim. This corresponds to Claim 1 in the proof of Theorem 4 in \cite{bianconiundef}.

\begin{claim}\label{claim-undefclaim1}
	The function $u(x,y)$ is not definable in $\Rbar$. In fact the functions $x,y,u(x,y),\\v(x,y)$ are algebraically independent over $\R$.
\end{claim}

By applying the addition formula for $\wp$ we may translate and shrink the interval $I$ and assume that $I\subseteq [0,1]$. Similarly we may replace $D$ with a smaller disc and assume that $D\subseteq I^2\subseteq[0,1]^2$. If $f$ is algebraic on this smaller disc it will be algebraic on the original disc and it therefore suffices to prove the theorem on the smaller disc. The images of $u$ and $v$ restricted to this disc will be bounded and by a final translating and scaling we may suppose that these images are contained in the interval $I$. 
\par
Let $f_2(x,y)=u(x,y)$ and $f_3(x,y)=v(x,y)$. By Proposition \ref{proposition- gab imp definition }, for some integer $n\ge1$ and an open box $B\subseteq D$ there are polynomials $P_2,\dots,P_n\in\R[y_0,\dots,y_{3n+2}]$ and non-zero rationals $a_0,\dots,a_n$, certain functions $f_4,\dots,f_n:B\rightarrow I$, such that for all $(x,y)\in B$,

$$\begin{array}{ccc}
F_2(x,y,f_2(x,y),\dots,f_n(x,y))=0\\\vdots\\F_n(x,y,f_2(x,y),\dots,f_n(x,y))=0
\end{array}$$	
and
$$\det\left(\frac{\partial F_i}{\partial x_j}\right)_{\substack{i=2,\dots,n\\j=2,\dots,n}}(x,y,f_2(x,y),\dots,f_n(x,y))\ne 0,$$
where for $i=2,\dots,n$ we have that 
$$
F_i(x_0,\dots,x_n)=P_i(x_0,\dots,x_n,\wp(a_0x_0),\dots,\wp(a_nx_n),\wp'(a_0x_0),\dots,\wp'(a_nx_n)).
$$
Therefore for all $i,j=2,\dots,n$

\begin{equation}\label{equation- differentiate F_i}\frac{\partial F_i}{\partial x_j}(x_0,\dots,x_n)=\frac{\partial P_i}{\partial y_j}(\ybar)+a_j\wp'(a_jy_j)\frac{\partial P_i}{\partial y_{j+n+1}}(\ybar)+a_j\wp''(a_jy_j)\frac{\partial P_i}{\partial y_{j+2n+2}}(\ybar),   \end{equation}
where 

$$\ybar=(x_0,\dots,x_n,\wp(a_0x_0),\dots,\wp(a_nx_n),\wp'(a_0x_0),\dots,\wp(a_nx_n)).$$

Let $f_0(x,y)=x$ and $f_1(x,y)=y$. Now $n$ is taken to be minimal such that there exists an open box $B$, some non-zero rationals $a_0,\dots,a_n$ and polynomials $P_2,\dots,P_n$ in $3n+3$ variables and $F_i(x_0,\dots,x_n)=P_i(x_0,\dots,x_n,\wp(a_0x_0),\dots,\wp(a_nx_n),\wp'(a_0x_0),\dots,\wp'(a_nx_n))$ and there are also some functions $f_4,\dots,f_n$ whose domain is $B$ such that $F_i(f_0(x,y),\dots,f_n(x,y))=0$ and $\det(\partial F_i/\partial x_j)(f_0(x,y),\dots,f_n(x,y))\ne 0$ for all $(x,y)\in B$. The functions $f_0,\dots,f_n$ are real analytic on a disc $D'\subseteq B$ centred at some $\alpha=(\alpha_1,\alpha_2)\in B$. It can easily be shown that $f_0-f_0(\alpha),\dots,f_n-f_n(\alpha)$ are linearly independent over $\Q$. Applying Theorem \ref{theorem-ax-schanuel for P-analytic functions on disc- multiple P functions} to $a_0f_0,\dots,a_nf_n$ gives that

$$\trdeg_\C\C[f_0,\dots,f_n,\wp(a_0f_0),\dots,\wp(a_nf_n)]\ge n+2.$$

The rest of the proof consists of finding a contradictory upper bound on this transcendence degree. Let $$\xtilde=\xtilde(x,y)=(f_0(x,y),\dots,f_n(x,y))$$ and \begin{align*}
\ytilde=\ytilde(x,y)=(&f_0(x,y),\dots,f_n(x,y),\wp(a_0f_0(x,y)),\dots,\wp(a_nf_n(x,y)),\\&\wp'(a_0f_0(x,y)),\dots,\wp'(a_nf_n(x,y)))
\end{align*}
for all $(x,y)\in B$. From \eqref{equation- differentiate F_i} it is clear that for all $(x,y)\in B$

$$	\begin{pmatrix}
\frac{\partial F_2}{\partial x_2}&\dots&\frac{\partial F_2}{\partial x_n}\\\vdots&\ddots&\vdots\\\frac{\partial F_n}{\partial x_2}&\dots&\frac{\partial f_n}{\partial x_n}
\end{pmatrix}(\xtilde(x,y))=\begin{pmatrix}
\frac{\partial P_2}{\partial y_2}&\dots&\frac{\partial P_2}{\partial y_{3n+2}}\\\vdots&\ddots&\vdots\\\frac{\partial P_n}{\partial y_2}&\dots&\frac{\partial P_n}{\partial y_{3n+2}}
\end{pmatrix}(\ytilde(x,y))\cdot M,
$$
where $M$ is the $(3n+1)\times(n-1)$ matrix

$$M=\begin{pmatrix}
&0&0&&0&0&\\I_{n-1}&\vdots&\vdots& M_1&\vdots&\vdots& M_2\\&0&0&&0&0&
\end{pmatrix}^T$$ where 

$$M_1=\begin{pmatrix}
a_2\wp'(a_2f_2(x,y))&\dots&0\\\vdots&\ddots&\vdots\\0&\dots&a_n\wp'(a_nf_n(x,y))
\end{pmatrix}$$ and 

$$M_2=\begin{pmatrix}
a_2\wp''(a_2f_2(x,y))&\dots&0\\\vdots&\ddots&\vdots\\0&\dots&a_n\wp''(a_nf_n(x,y))
\end{pmatrix}.$$
The rows of 

$$	\begin{pmatrix}
\partial F_2/\partial x_2&\dots&\partial F_2/\partial x_n\\\vdots&\ddots&\vdots\\\partial F_n/\partial x_2&\dots&\partial F_n/\partial x_n
\end{pmatrix}(\xtilde(x,y))$$ are linearly independent over $\R$ and so the rows of 

$$\begin{pmatrix}
\partial P_2/\partial y_2&\dots&\partial P_2/\partial y_{3n+2}\\\vdots&\ddots&\vdots\\\partial P_n/\partial y_2&\dots&\partial P_n/\partial y_{3n+2}
\end{pmatrix}(\ytilde(x,y))$$
are also linearly independent over $\R$. Therefore for all $(x,y)\in B$ the matrix 

$$\begin{pmatrix}
\partial P_2/\partial y_2&\dots&\partial P_2/\partial y_{3n+2}\\\vdots&\ddots&\vdots\\\partial P_n/\partial y_2&\dots&\partial P_n/\partial y_{3n+2}
\end{pmatrix}(\ytilde(x,y))$$
 has maximal rank $n-1$. Given Proposition 5.3 in Chapter 8 of \cite{langalg} it follows by a standard argument that

$$\trdeg_\C \C[f_0,\dots,f_n,\wp(a_0f_0),\dots,\wp(a_nf_n)]\le 2n+4.$$
 In order to obtain the desired contradictory upper bound $n+3$ polynomial equations shall be added to the system and it shall be shown how this lowers the upper bound on transcendence degree. The first $n+1$ of these equations correspond to the differential equation for the $\wp$-function in each of the $n+1$ variables and the final two of these equations arises from the Cauchy-Riemann equations for the functions $u$ and $v$. For each $i=0,\dots,n$ define $$P_{i+n+1}(y_{i+n+1},y_{i+2n+2})=y_{i+2n+2}^2-4y_{i+n+1}^3+g_2y_{i+n+1}+g_3.$$ For all $(x,y)\in B$ and $i=0,\dots,n$ $$P_{i+n+1}(\wp(a_if_i(x,y)),\wp'(a_if_i(x,y)))=0.$$
By differentiating and using \eqref{equation- P'' formula} it can be shown that for all $i=0,\dots,n$ and $(x,y)\in B$,
\begin{align*}
\frac{\partial P_{i+n+1}}{\partial y_j}(y_{j+n+1},y_{j+2n+2})&+a_j\wp'(a_jf_j(x,y))\frac{\partial P_{i+n+1}}{\partial y_{j+n+1}}(y_{j+n+1},y_{j+2n+2})\\&+a_j\wp''(a_jf_j(x,y))\frac{\partial P_{i+n+1}}{\partial y_{j+2n+2}}(y_{j+n+1},y_{j+2n+2})=0.
\end{align*}
It can then easily be shown that the matrix

 $$\begin{pmatrix}
 	\partial P_2/\partial y_2&\dots&\partial P_2/\partial y_{3n+2}\\\vdots&\ddots&\vdots\\\partial P_{2n+1}/\partial y_2&\dots&\partial P_{2n+1}/\partial y_{3n+2}
 \end{pmatrix}(\ytilde(x,y))$$
 has maximal rank $2n$ and therefore by the same standard argument we have that 

$$\trdeg_\C \C[f_0,\dots,f_n,\wp(a_0f_0),\dots,\wp(a_nf_n),\wp'(a_0f_0),\dots,\wp'(a_nf_n)]\le n+3.$$

By the implicit function theorem the derivatives of $f_i(x_0,x_1)$ for $i=2,\dots,n$ are given by 

$$\begin{pmatrix}
\frac{\partial f_2}{\partial x_k}\\\vdots\\\frac{\partial f_n}{\partial x_k}
\end{pmatrix}=-\Delta^{-1}\begin{pmatrix}
\frac{\partial F_2}{\partial x_k}\\\vdots\\\frac{\partial F_n}{\partial x_k}
\end{pmatrix},$$
where $k=0,1$ and $\Delta=(\partial F_i/\partial x_j)$ and the right hand side is evaluated at $(x_0,\dots,x_n)=(f_0,\dots,f_n)$. Multiplying both sides by the determinant of $\Delta$ and using the Cauchy-Riemann equations for $f_2$ and $f_3$ gives two new equations $F_0$ and $F_1$ with corresponding polynomials $P_0$ and $P_1$, following the method of Bianconi in \cite{bianconiundef}. These are of the form,

\begin{align*}
F_0=[& \text{first line of }-\det\Delta\cdot(\Delta^{-1}(\partial F_i/\partial x_0))\\& \text{ minus the second line of }-\det\Delta\cdot(\Delta^{-1}(\partial F_i/\partial x_1))]
\end{align*}
and 
\begin{align*}
F_1=[& \text{first line of }-\det\Delta\cdot(\Delta^{-1}(\partial F_i/\partial x_1))\\& \text{ plus the second line of }-\det\Delta\cdot(\Delta^{-1}(\partial F_i/\partial x_0))].
\end{align*}
In order to lower the upper bound further we have the following lemma, the proof of which adapts those of Claims 5 and 6 in the proof of Theorem 4 in \cite{bianconiundef}.

\begin{lemma}\label{lemma- Undef- perturbation lemma}
	For each $k=0,1$ there is a point $z\in \C^{3n+3}$ such that $P_k(z)\ne0$ and $P_{1-k}(z)=0$ and $P_i(z)=0$ for all $i=2,\dots,2n+1$.
\end{lemma}
\begin{proof}
	This adapts the proofs of Claims 5 and 6 in the proof of Theorem 4 in \cite{bianconiundef}. Let $V$ be the subset of $\R^{3n+3}$ defined by 
	
	$$V=\{(x,y,z)\in\R^{n+1}\times\R^{n+1}\times\R^{n+1}:y=\wp(ax),z=\wp'(ax) \}$$
	\noindent
	where $\wp(ax)=(\wp(a_0x_0),\dots,\wp(a_nx_n))$ and $\wp'(ax)=(\wp'(a_0x_0),\dots,\wp'(a_nx_n))$. Also let $W$ be the subset of $\R^{3n+3}$ defined by 
	
	\begin{align*}W=\{&z\in\R^{3n+3}:P_2(z)=0,\dots,P_{2n+1}(z)=0 \text{ and }(\partial P_i/\partial y_j)(z)\ne 0\\& \text{ for }i=2,\dots,2n+1,j=2,\dots,3n+2 \text{ has maximal rank }\}.\end{align*}
	
	Let $X$ be the subset of $\R^{3n+3}$ defined by $\{ \ytilde(x,y)|(x,y)\in B \}$. Then it is clear that $X\subseteq V\cap W$.

	The subset $V$ may also be written as $$V=\{ (x,y,z)\in\R^{n+1}\times \R^{n+1}\times \R^{n+1}: \hat{F}_0(x,y,z)=\dots=\hat{F}_{2n+1}(x,y,z)=0 \},$$ where for $i=0,\dots,n$
	
	\begin{align*}
	&\Fhat_i(x,y,z)=y_i-\wp(a_ix_i)
	\\
	&\Fhat_{i+n+1}(x,y,z)=z_i-\wp'(a_ix_i).
	\end{align*}
	
	We denote the Jacobian matrix for this system by $\Phi$ and this is a $(2n+2)\times (3n+3)$ matrix given by
	
	$$\Phi=\begin{pmatrix}
	-a_0\wp'(a_0x_0)&\dots&0&1&\dots&0&0&\dots&0\\
	\vdots&\ddots&\vdots&\vdots&\ddots&\vdots&\vdots&\ddots&\vdots
	\\
	0&\dots&-a_n\wp'(a_nx_n)&0&\dots&1&0&\dots&0
	\\
	-a_0\wp''(a_0x_0)&\dots&0&0&\dots&0&1&\dots&0\\
	\vdots&\ddots&\vdots&\vdots&\ddots&\vdots&\vdots&\ddots&\vdots
	\\
	0&\dots&-a_n\wp''(a_nx_n)&0&\dots&0&0&\dots&1
	\end{pmatrix}.$$
	
	The normal space to $V$ at a point is generated by the rows of $\Phi$ evaluated at this point. Recall the matrix $M$,

	$$M=\begin{pmatrix}
	& 0&0& &0&0&\\I_{n-1}&\vdots&\vdots&M_1&\vdots&\vdots&M_2\\& 0&0&&0&0&
	\end{pmatrix}^T$$
	\noindent
	where 
	
	$$M_1=\begin{pmatrix}
	a_2\wp'(a_2f_2(x,y))&\dots&0\\\vdots&\ddots&\vdots\\0&\dots&a_n\wp'(a_nf_n(x,y))
	\end{pmatrix}$$
	\noindent
	and 
	
	$$M_2=\begin{pmatrix}
	a_2\wp''(a_2f_2(x,y))&\dots&0\\\vdots&\ddots&\vdots\\0&\dots&a_n\wp''(a_nf_n(x,y))
	\end{pmatrix}.$$

	Let $M'$ be the matrix 
	
	$$M'=\begin{pmatrix}
	0&0&\\\vdots&\vdots&M^T\\0&0&
	\end{pmatrix}.$$

	Then the matrix product $M'\cdot(\Phi(\ytilde))^T$ gives the $(n-1)\times (2n+2)$ zero matrix. Therefore the kernel of the linear transformation from $\R^{3n+3}$ to $\R^{2n+2}$ given by the matrix $M'$ is generated by the rows of the matrix $\Phi(\ytilde)$. Let $P$ be the matrix 
	
	$$P=\begin{pmatrix}
	\partial P_2/\partial y_0&\dots&\partial P_n/\partial y_0\\\vdots&\ddots&\vdots\\\partial P_2/\partial y_{3n+2}&\dots&\partial P_n/\partial y_{3n+2}
	\end{pmatrix}(\tilde{y}).$$

	Then we have that 
	
	$$M'\cdot P=\begin{pmatrix}
	\partial F_2/\partial x_2&\dots&\partial F_n/\partial x_2\\\vdots&\ddots&\vdots\\\partial F_2/\partial x_n&\dots&\partial F_n/\partial x_n
	\end{pmatrix}(\tilde{x}).$$

	
	The columns of the matrix on the right hand side of this equation are linearly independent over $\R$. Therefore the subspace of $\R^{3n+3}$ generated by the columns of $P$ has trivial intersection with the kernel of the linear transformation given by $M'.$ As the normal space to $W$ at a point is generated by the columns of $P$ evaluated at this point we have that in particular the normal spaces to $V$ and $W$ at each point in $X$ have trivial intersection and so the intersection of $V$ and $W$ is transversal. 
	
	\par 
	
	Therefore if the subspace $V$ is shifted locally then the intersection of $V$ and $W$ is still transversal. We shall now give such a shift explicitly. For real numbers $\eta$ and $\xi$ we let $V_{\eta,\xi}$ be the subset given by applying the following operations to $V$. In other words $V_{\eta,\xi}=\Psi(V)$ for $ \Psi:\R^{3n+3}\rightarrow\R^{3n+3}$ where $\Psi$ does the following, for $(y_0,\dots,y_{3n+2})\in\R^{3n+3}$  
	
	\begin{align*}
	& y_2\mapsto y_2+\eta y_0+\xi y_1
	\\& y_{2+n+1}\mapsto\frac{1}{4}\left(\frac{y_{2+2n+2}-\wp^\prime(a_2(\eta y_0+\xi y_1))}{y_{2+n+1}-\wp(a_2(\eta y_0+\xi y_1))}\right)^2-y_{2+n+1}-\wp(a_2(\eta y_0+\xi y_1)) 
	\end{align*}
	and
	\begin{align*}&y_{2+2n+2}\mapsto \frac{\begin{matrix}
		\Big( \wp(a_2(\eta y_0+\xi y_1))y_{2+2n+2}-\wp'(a_2(\eta y_0+\xi y_1))y_{2+n+1}\\-\wp(a_2(y_2+\eta y_0+\xi y_1))(y_{2+2n+2}-\wp'(a_2(\eta y_0+\xi y_1))\Big)
		\end{matrix}}{y_{2+n+1}-\wp(a_2(\eta y_0+\xi y_1))}
	\end{align*}
	\noindent
	and the rest of the variables are fixed. The projection of $W$ onto the variables $y_0,y_1,y_2,y_3$ contains the set 
	
	$$\{(f_0,f_1,f_2(f_0,f_1),f_3(f_0,f_1))|f_0,f_1\in B \}$$
	\noindent
	in its interior. If it did not then as $\dim \pi W=4$ we have $\dim \partial W\le3$ and so there is an algebraic relation between $f_0,f_1,f_2$ and $f_3$ contradicting Claim \ref{claim-undefclaim1}. So for each real $\eta$ and $\xi$ there is a positive real number $\delta$ such that for all real $f_0$ and $f_1$ with $f_0^2+f_1^2<\delta^2$ the intersection of $X$ with $V_{\eta,\xi}$ is non-empty. The effect of $\Psi$ on the subset $X$ is the following.
	
	\begin{align*}
	&f_2\rightarrow f_2+\eta f_0+\xi f_1
	\\ & \wp(a_2f_2)\rightarrow \wp(a_2(f_2+\eta f_0+\xi f_1))
	\\& \wp'(a_2f_2)\rightarrow \wp'(a_2(f_2+\eta f_0+\xi f_1)).
	\end{align*}
	
	
	The real numbers $\eta$ and $\xi$ may be chosen so that at least one of the Cauchy-Riemann equations for $u$ and $v$ are not satisfied. Therefore there is a point $z\in\R^{3n+3}$ such that $P_k(z)\ne 0$ for some $k=0,1$ and $P_{1-k}(z)=P_j(z)=0$ for $j=2,\dots,2n+1$ and so the lemma is proved.
\end{proof}

By shrinking and shifting the disc $D$ if necessary we may assume that all the points

\begin{align*}
\ytilde(x,y)=(&x,y,f_2(x,y),\dots,f_n(x,y),\\&\wp(a_0x),\wp(a_1y),\wp(a_2f_2(x,y)),
\dots,\wp(a_nf_n(x,y)),\\&\wp'(a_0x),\wp'(a_1y),\wp'(a_2f_2(x,y)),\dots,\wp'(a_nf_n(x,y)))
\end{align*}
\noindent
such that the system $P_2(\ytilde)=\dots=P_{2n+1}(\ytilde)=0$ is satisfied are contained in a single irreducible component of the variety $\Vcal(\langle P_2,\dots,P_{2n+1}\rangle)$ denoted $\mathcal{W}$. Suppose that $\dim (\mathcal{W}\cap \Vcal(\langle P_0\rangle))=\dim \mathcal{W}$. Then $\Wcal\cap\Vcal(\langle P_0\rangle )=\mathcal{W}$ as $\mathcal{W}$ is irreducible. By the proof of Lemma \ref{lemma- Undef- perturbation lemma} there is a point $z\in\Wcal$ such that $P_2(z)=\dots=P_{2n+1}(z)=0$ and $P_0(z)\ne 0$. Therefore there is a point $z\in\Wcal$ such that $z\notin \Vcal(P_0)$, a contradiction. By once again shifting and shrinking the disc $D$ we may suppose that all of the points $\ytilde(x,y)$ satisfying the system $P_0(\ytilde)=P_2(\ytilde)=\dots=P_{2n+1}(\ytilde)=0$ are contained in an irreducible component of the variety $\V(\langle P_0,P_2,\dots,P_{2n+1}\rangle)$, denoted $\Wcal'$.

\par 
Suppose that $\dim ( \Wcal'\cap \Vcal(\langle P_1\rangle))=\dim \Wcal'$, then again as $\Wcal'$ is irreducible we have that $\Wcal'\cap\Vcal(\langle P_1\rangle )=\Wcal'$. Again by the proof of Lemma \ref{lemma- Undef- perturbation lemma} there is a point $z\in\Wcal$ such that only one of $P_0(z)$ and $P_1(z)$ equals zero and $P_2(z)=\dots=P_{2n+1}(z)=0$. Therefore there is a point $z\in \Wcal'$ and $z\notin \V(\langle P_1\rangle)$, a contradiction as required. We have shown that if we add each of the polynomials $P_0$ and $P_1$ to the system $P_2,\dots,P_{2n+1}$ and consider the variety corresponding to the ideal generated by each of these new systems in turn then the dimension of each of these varieties decreases. Hence the upper bound on the transcendence degree of our finitely generated extension of $\C$ decreases by two.
\par 
Therefore we have a lower bound $$\trdeg_\C \C[f_0,\dots,f_n,\wp(a_0f_0),\dots,\wp(a_nf_n)]\ge n+2$$ and an upper bound  $$\trdeg_\C \C[f_0,\dots,f_n,\wp(a_0f_0),\dots,\wp(a_nf_n)]\le n+1,$$ a contradiction as required.
\par 
If $\Omega$ is not a real lattice then one must consider the structure $(\Rbar,\Re(\wp)|_I,\Im(\wp)|_I,\\\Re(\wp')|_I,\Im(\wp')|_I)$, which is also model complete by Gabrielov's result, Corollary \ref{corollary-background-gabrielov- model completeness}. The presence of the real and imaginary parts of $\wp$ gives an extra $2n+2$ variables in the system of polynomial equations arising from Proposition \ref{proposition- gab imp definition }. This raises the corresponding upper bound by $2n+2$. Therefore the method in the real lattice case must be adapted in order to find the required contradictory upper and lower bounds on transcendence degree. By Proposition \ref{proposition- gab imp definition } we have a system of polynomials involving the real and imaginary parts of both $\wp$ and $\wp'$, which may be rearranged to give a polynomial system involving $\wp,\wp',\wptilde$ and $\wptilde'$ where $\wptilde(z)=\overline{\wp(\zbar)}=\wp_{\Omegabar}(z)$. If $\Omega$ is not isogenous to $\Omegabar$ then there are no integers $a,b,c,d$ with $ad-bc\ne0$ such that $\overline\tau=(a\tau+b)/(c\tau+d)$ and so we may apply Theorem \ref{theorem-ax-schanuel for P-analytic functions on disc- multiple P functions} with the Weierstrass functions $\wp$ and $\wptilde$ in order to obtain a higher lower bound on transcendence degree. In order to lower the corresponding upper bound on transcendence degree further we add polynomial equations corresponding to the differential equation for $\wptilde$ in each variable as well as corresponding versions of the polynomial equations added in the real lattice case. This gives the desired contradiction. \par  If $\Omega$ is isogenous to its complex conjugate then there is a non-zero complex number $\alpha$ such that $\alpha\Omega\subseteq\Omegabar$. Therefore from the definition of $\wp$ we may rewrite $\wptilde(z)$ as a rational function in $\wp(\alpha^{-1}z)$. The system of polynomials obtained using Proposition \ref{proposition- gab imp definition } may be rewritten as system of rational functions involving $\wp(z),\wp(\alpha^{-1}z),\wp'(z)$ and $\wp'(\alpha^{-1}z)$ from which a system of polynomials may be obtained. The lower bound on transcendence degree is raised by applying Theorem \ref{theorem-ax-schanuel for P-analytic functions on disc- multiple P functions} with $\wp$ to the functions $a_0f_0,\dots,a_nf_n,\alpha^{-1}a_0f_0,\dots,\alpha^{-1}a_nf_n$. The upper bound on transcendence degree is lowered further by adding polynomial equations corresponding to the differential equation for $\wp(\alpha^{-1}z)$ in each variable as well as once again adding corresponding versions of the polynomial equations added in the real lattice case. This completes the proof of Theorem \ref{theorem- undef for P theorem}.

\section{Further definability results}\label{section-further results}
The first result in this section is an immediate corollary of Theorem \ref{theorem- undef for P theorem} combined with Theorem 12.5 in \cite{newtonpetstar}.

\begin{corollary}
	Let $\Omega\subseteq\C$ be a complex lattice which does not have complex multiplication and $I$ be a bounded closed interval in $\R$ which does not intersect $\Omega$. Let $X$ be an analytic subset of an open set $U\subseteq\C^n$. Assume that $U$ and $X$ are definable in $(\Rbar,\wp|_I)$. Then there is a complex algebraic set $A\subseteq\C^n$ such that $X$ is an irreducible component of $A\cap U$.
\end{corollary}

For real lattices the following theorem can be seen in \cite{mcculloch2020nondefinability}. Here the result in \cite{mcculloch2020nondefinability} is extended to all complex lattices and a different proof is given.
\begin{theorem}\label{theorem-nondef for P}
	Let $\Omega$ be a complex lattice and $I\subseteq\R$ a bounded closed interval such that $I\cap\Omega$ is empty. Let $D\subseteq\C$ be a disc. Then $\wp|_D$ is definable in $(\Rbar,\wp|_I)$ if and only if the lattice $\Omega$ has complex multiplication.
\end{theorem}
\begin{proof}
	Suppose that $D\cap\Omega$ is empty. Firstly we assume that $\Omega$ has complex multiplication and so there is a non-zero complex number $\alpha$ such that $\alpha\Omega\subseteq\Omega$. Define $f(z)=\wp(\alpha z)$. Then for all $\omega\in\Omega$ we have that 
	
	$$f(z+\omega)=\wp(\alpha z+\alpha\omega)=\wp(\alpha z)$$
	 and so $f$ is a meromorphic function that is periodic with respect to $\Omega$. By Theorem 3.2 in Chapter 6 of \cite{silverman} the function $f$ is a rational function in terms of $\wp$ and $\wp'$. Therefore $\wp|_{\alpha I}$ is definable in $(\Rbar,\wp|_I)$. Similarly we have that $\wp'|_{\alpha I}$ is definable in $(\Rbar,\wp|_I)$. We may assume that $D\subseteq I \times \alpha I$. Therefore for any $z\in D$ we have that $z=x+\alpha y$ for $x,y\in I$. By the addition formula for $\wp$ 
	 
	 $$\wp(x+\alpha y)=R(\wp(x),\wp'(x),\wp(y),\wp'(y))$$
	  for a rational function $R$. Therefore $\wp|_D$ is definable in $(\Rbar,\wp|_I)$. Conversely, suppose that $\Omega$ does not have complex multiplication and that there is a disc $D\subseteq\C$ such that $\wp|_D$ is definable in $(\Rbar,\wp|_I)$. As $\wp$ is holomorphic on $D$ we have that by Theorem \ref{theorem- undef for P theorem} the function $\wp|_D$ is definable in $\Rbar$, a contradiction.
	\par 
	Now let $D$ be a disc containing a single lattice point $\omega\in\Omega$ and consider the function $f(z)=(z-\omega)^2\wp(z)$. If $\Omega$ has complex multiplication then as $(z-\omega)^2 $ is definable in the structure $(\Rbar,\wp|_I)$ it is clear by a repetition of the above argument we have that $f|_D$ is definable in $(\Rbar,\wp|_I)$. Conversely suppose that $\Omega$ does not have complex multiplication and assume for a contradiction that $f|_D$ is definable in the structure $(\Rbar,\wp|_I)$. Then $f|_{D'}$ is definable in $(\Rbar,\wp|_I)$ for some disc $D'\subseteq D$ that does not contain $\omega$. Therefore $\wp|_{D'}$ is definable in $(\Rbar,\wp|_I)$, a contradiction.   
\end{proof}

In the proof of Theorem \ref{theorem- undef for P theorem} the existence of an Ax-Schanuel statement for the Weierstrass $\wp$-function is essential. This raises the question of whether we can recover corresponding nondefinability results for other transcendental functions that also satisfy an Ax-Schanuel theorem. In this context the modular $j$-function is a natural function to consider and the Ax-Schanuel result is due to Pila and Tsimerman in \cite{axforj}. The following theorem can be thought of as a $j$-function analogue of Theorem \ref{theorem-nondef for P}. The proof of this theorem adapts a similar method to the one seen in Section \ref{section-proof of theorem} and uses the first implicit definition in Section \ref{section-implicit definition}.
\begin{theorem}\label{nondef j theorem}
	Let $I\subseteq \R^{>0}$ be an open interval that is bounded away from zero and let $D\subseteq\Hbb$ be a non-empty disc. Then the restriction of $j$ to the disc $D$ is not definable in the structure $(\Rbar,j|_{iI})$.
\end{theorem}
\begin{proof}
	Assume for a contradiction that there is a disc $D\subseteq\Hbb$ such that the restriction $j|_D$ is definable in the structure $(\Rbar,j|_{iI})$. For notational convenience we can suppose that the disc $D$ contains the horizontal line segment $i+I$ and so the real and imaginary parts of the function $j|_{i+I}$ are definable in the structure $(\Rbar,j|_{iI})$. Rearranging the differential equation satisfied by $j$ given in \eqref{differential equation for j} gives that 
	
	\begin{equation}
	ij'''(it)=\frac{-3}{2}\frac{(j''(it))^2}{ij'(it)}+\left(\frac{j^2(it)-1968j(it)+2654208}{2j^2(it)(j(it)-1728)^2}\right)(ij'(it))^3
	\end{equation}
	and so $ij'''(it)$ may be written as a polynomial in $j(it),ij'(it),j''(it),(ij'(it))^{-1}$ and \\$(2j^2(it)(j(it)-1728)^2)^{-1}$. By shrinking the interval $I$ if necessary we may assume that the denominators do not vanish for any $t\in I$. Therefore by differentiating this equation with respect to $t$ we can see that all the higher derivatives of $j(it)$ may also be given as polynomials in these functions.
	Consider the auxiliary structure given by expanding $\Rbar$ by the functions $j_B(t)=j(iB(t)),j_B'(t)=ij'(iB(t)),j_B''(t)=j''(iB(t)),j_1(t)=(ij'(B(t)))^{-1}$ and $j_2(t)=(2j(iB(t))^2(j(iB(t))-1728)^2)^{-1}$ as well as $B$ and $B_1$. Here $B:\R\rightarrow I$ is an algebraic function and $B_1$ is a rational function arising from the derivative of $B$ such that all higher derivatives of $B$ are polynomials in $B$ and $B_1$. The structures $(\Rbar,j|_{iI})$ and $(\Rbar,j_B,j'_B,j''_B,j_1,j_2,B,B_1)$ are equivalent in the sense of having the same definable sets. They also have the same universally and existentially definable sets. Therefore the real and imaginary parts of the function $j|_{i+I}$ are definable in the structure $(\Rbar,j_B,j'_B,j''_B,j_1,j_2,B,B_1)$. Therefore it suffices to prove Theorem \ref{nondef j theorem} in this auxiliary structure. It is clear from construction that the set $\{ j_B,j'_B,j''_B,j_1,j_2,B,B_1 \}$ is closed under differentiation and the ring of terms of this auxiliary structure is closed under differentiation in the sense of Section \ref{section-implicit definition}. By the Gabrielov result, Corollary \ref{corollary-background-gabrielov- model completeness}, the auxiliary structure $(\Rbar,j_B,j'_B,j''_B,j_1,j_2,B,B_1)$ is model complete.
	\par 
	Let $f_1,f_2:I\rightarrow \R$ be defined by $f_1(t)=\Re(j(i+t))$ and $f_2(t)=\Im(j(i+t))$. By applying Theorem \ref{desingularisation theorem} to both $f_1$ and $f_2$, we have that for some integer $n\ge 1$ and a subinterval $I'\subseteq I$ there are polynomials $P^*_1,\dots,P^*_{n}:\R^{8n+8}\rightarrow \R$ in $\R[y_1,\dots,y_{8n+8}]$, certain functions $f_3,\dots,f_{n}:I'\rightarrow \R$ such that for all $t\in I'$,
	
	$$\begin{array}{ccc}
	F_1(t,f_1(t),\dots,f_{n}(t))=0\\
	\vdots\\
	F_n(t,f_1(t),\dots,f_{n}(t))=0\end{array}$$
	\noindent
	and
	
	$$\det \left(\frac{\partial F_i}{\partial x_j}\right)_{\substack{i=1,\dots,n\\j=2,\dots,n+1}}(t,f_1(t),\dots,f_{n}(t))\ne 0,$$
	\noindent
	where for $i=1,\dots,n$ we have that 
	
	\begin{align*}
	F_i(t,f_1(t),\dots,f_n(t))=P^*_i(&t,f_1(t),\dots,f_{n}(t),
	\\& j(iB(t)),j(iB(f_1(t))),\dots,j(iB(f_{n}(t))),
	\\& ij'(iB(t)),ij'(iB(f_1(t))),\dots,ij'(iB(f_{n}(t))),
	\\& j''(iB(t)),j''(iB(f_1(t))),\dots,
	j''(iB(f_{n}(t))),
	\\& j_1(t),j_1(f_1(t)),\dots,j_1(f_n(t)),
	\\&j_2(t),j_2(f_1(t)),\dots,j_2(f_n(t))
	\\& B(t),B(f_1(t)),\dots,B(f_{n}(t)),
	\\& B_1(t),B_1(f_1(t)),\dots,B_1(f_{n}(t))).
	\end{align*}
	
	By the definition of the functions $j_1$ and $j_2$ as well as $B$ and $B_1$ we may write $F_1,\dots,F_n$ as algebraic functions in $t,f_1(t),\dots,f_{n}(t),j(iB(t)),j(iB(f_1(t))),\dots,j(iB(f_{n}(t)))$ and \\$ij'(iB(t)),ij'(iB(f_1(t))),\dots,ij'(iB(f_{n}(t)))$ as well as $j''(iB(t)),j''(iB(f_1(t))),\dots,\\j''(iB(f_{n}(t)))$. In defining these algebraic functions square roots are introduced from the definition of $B$, which may affect the analyticity of these algebraic functions. The domain of these algebraic functions is a small open subset of $\R^{4n+4}$ containing the set 
	
	$$\Gamma_j=\{[f(t),j(iB(f(t))),ij'(iB(f(t))),j''(iB(f(t)))]:t\in I' \}$$
	 where $f(t)=(t,f_1(t),\dots,f_n(t))$ and the algebraic functions are taken to be analytic on this domain. Hence for $i=1,\dots,n$ we have that
	
	\begin{align*}
	F_i(x_1,\dots,x_{n+1})=P_i(&x_1,\dots,x_{n+1},j(iB(x_1)),\dots,j(iB(x_{n+1})),\\& ij'(iB(x_1)),\dots,ij'(iB(x_{n+1})),j''(iB(x_1)),\dots, j''(iB(x_{n+1})))
	\end{align*}
	\noindent
	for algebraic functions $P_1,\dots,P_n$ and in particular for all $t\in I'$,
	
	\begin{align*}F_i(t,f_1(t),\dots,f_{n}(t))=P_i[&t,f_1(t),\dots,f_{n}(t),
	\\& j(iB(t)),j(iB(f_1(t))),\dots,j(iB(f_{n}(t))),
	\\&  ij'(iB(t)), ij'(iB(f_1(t))),\dots,ij'(iB(f_{n}(t))),
	\\& j''(iB(t)),
	j''(iB(f_1(t))),\dots, j''(iB(f_{n}(t)))]=0.
	\end{align*}
	Now take $n$ to be minimal such that the subinterval $I'$, the functions $f_3,\dots,f_n$ and the system of algebraic functions $P_1,\dots,P_n$ exists as given above. Let \begin{align*}
	\ytilde=\ytilde(t)=(&t,f_1(t),\dots,f_{n}(t),
	\\& j(iB(t)),j(iB(f_1(t))),\dots,j(iB(f_{n}(t))),
	\\&  ij'(iB(t)), ij'(iB(f_1(t))),\dots,ij'(iB(f_{n}(t))),
	\\& j''(iB(t)),
	j''(iB(f_1(t))),\dots, j''(iB(f_{n}(t)))).	\end{align*} For all $t\in I'$ it can easily be shown that the matrix
	$$\left(
	\frac{\partial P_i}{\partial y_j}\right)_{\substack{i=1,\dots,n\\j=2,\dots,4n+4}}(\ytilde(t))$$
	has maximal rank $n$. The standard argument noted in the proof of Theorem \ref{theorem- undef for P theorem} can be readily adapted for a system of algebraic functions and so 
	
	\begin{align*}
	\trdeg_{\mathbb{C}}\mathbb{C}[&t,f_1,\dots,f_n,\\&j(iB(t)),j(iB(f_1)),\dots,j(iB(f_n)),\\& ij^\prime(iB(t)),ij^\prime(iB(f_1)),\dots,ij^\prime(iB(f_n)),\\& j^{\prime\prime}(iB(t)),j^{\prime\prime}(iB(f_1)),\dots,j^{\prime\prime}(iB(f_n))] \le 4n+4-n=3n+4.
	\end{align*}
	
	Suppose that there is some integer $M\ge 1$ such that $$\Phi_M[j(iB(f_{k}(t))),j(iB(f_l(t)))]=0$$ for all $t\in I'$, where $0\le k,l\le n$ and $k\ne l$ and $f_0(t)=t$. For convenience we assume that $k=n-1$ and $n=l$. Then $iB(f_n(t))$ may be written as a rational function in $iB(f_{n-1}(t))$. Rearranging the modular polynomial $\Phi_M$ gives that $j(iB(f_n(t)))$ may be written as an algebraic function in $j(iB(f_{n-1}(t)))$. Differentiating both sides of this equation and rearranging and repeating this process gives algebraic functions for $ij'(iB(f_n(t)))$ and $j''(iB(f_n(t)))$ in terms of $f_{n-1}(t),j(iB(f_{n-1}(t))),ij'(iB(f_{n-1}(t)))$ and $f_{n-1}(t),j(iB(f_{n-1}(t))),ij'(iB(f_{n-1}(t))),j''(iB(f_{n-1}(t)))$ respectively. Therefore the non-singular system of algebraic functions $P_1,\dots,P_n$ may be rearranged to give a system of algebraic functions in fewer variables. If this system is non-singular at the points $\ytilde(t)$ then there is a contradiction to the minimality of $n$. Therefore this system is assumed to be singular at these points. However this leads to a contradiction of the non-singularity of the original system and we may therefore conclude that no such integer $M\ge 1$ exists. From this it can be shown that there is no integer $M\ge 1$ such that $\Phi_M(j(iB(f_k(t))),j(iB(f_l(t))))=0$ for all $k,l=0,\dots,n$ with $k\ne l$. Applying Theorem \ref{theorem-ax schanuel for j-analytic functions} to $i+f_0,iB(f_0),\dots,iB(f_n)$ gives that
	
	\begin{align*}\trdeg_\C\C[&i+t,iB(t),iB(f_1),\dots,iB(f_n),\\&j(i+t),j(iB(t)),j(iB(f_1)),\dots, j(iB(f_n)), \\&j'(i+t),j'(iB(t)),j'(iB(f_1)),\dots,j'(iB(f_n)),\\&j''(i+t),j''(iB(t)), j''(iB(f_1)),\dots,j''(iB(f_n)) ]\ge 3n+7
	\end{align*}
	and so
	\begin{align*}\trdeg_\C\C[&i+t,iB(t),iB(f_1),\dots,iB(f_n),\\&j(i+t),j(iB(t)),j(iB(f_1)),\dots, j(iB(f_n)), \\&j'(iB(t)),j'(iB(f_1)),\dots,j'(iB(f_n)),\\&j''(iB(t)), j''(iB(f_1)),\dots,j''(iB(f_n)) ]\ge 3n+5.
	\end{align*}
	As $f_1,f_2$ are the real and imaginary parts of $j(i+t)$ and the function $B$ is algebraic and $i+t$ and $iB(t)$ are algebraically dependent we have that
	\begin{align*}\trdeg_\C\C[&i+t,iB(t),iB(f_1),\dots,iB(f_n),\\&j(i+t),j(iB(t)),j(iB(f_1)),\dots, j(iB(f_n)), \\&j'(iB(t)),j'(iB(f_1)),\dots,j'(iB(f_n)),\\&j''(iB(t)), j''(iB(f_1)),\dots,j''(iB(f_n)) ]\le 3n+4,
	\end{align*}
	a contradiction.
\end{proof}

\section{Final remarks}\label{section- final remarks}

It is reasonable to expect that further nondefinability results for transcendental functions such as the modular j-function can be obtained by adapting the methods given here. In particular an analogue of Theorem \ref{theorem- undef for P theorem} for the modular $j$-function is a natural statement. However there are some obstructions in directly applying the method of Section \ref{section-proof of theorem} to this case. Firstly the necessity for a system of algebraic functions requires a reworking of the final part of the proof of Theorem \ref{theorem- undef for P theorem} for such a system. Also the lack of addition formula for the modular $j$-function makes a direct application of the proof of Lemma \ref{lemma- Undef- perturbation lemma} impossible. 
\par 
However for the Weierstrass $\zeta$-function, a quasi-periodic meromorphic function related to $\wp$ by the formula $\zeta'=\wp$ some definability results can be readily obtained. By using classical formulae and an Ax-Schanuel statement involving $\wp$ and $\zeta$, which is also due to Brownawell and Kubota in \cite{bkpaper} one can characterise the definability of restrictions of $\zeta$ to a disc $D\subseteq\C$ in the structure $(\Rbar,\wp|_I,\zeta|_I)$, where $I\subseteq\R$ is a bounded closed interval such that $I\cap\Omega=\emptyset$, in terms of complex multiplication. This is an analogue of Theorem \ref{theorem-nondef for P} for the Weierstrass $\zeta$-function and the proof is simply another adaptation of the method seen in the proof of Theorem \ref{nondef j theorem}.

\bibliographystyle{abbrv}

\bibliography{Nondefinability_for_elliptic_and_modular_functions} 

\begin{thebibliography}{10}

\bibitem{bianconimc}
R.~Bianconi.
\newblock Model completeness results for elliptic and abelian functions.
\newblock {\em Ann. Pure Appl. Logic}, 54(2):121--136, 1991.

\bibitem{bianconi}
R.~Bianconi.
\newblock Nondefinability results for expansions of the field of real numbers
  by the exponential function and by the restricted sine function.
\newblock {\em J. Symbolic Logic}, 62(4):1173--1178, 1997.

\bibitem{bianconiundef}
R.~Bianconi.
\newblock Undefinability results in o-minimal expansions of the real numbers.
\newblock {\em Ann. Pure Appl. Logic}, 134(1):43--51, 2005.

\bibitem{bkpaper}
W.~D. Brownawell and K.~K. Kubota.
\newblock The algebraic independence of {W}eierstrass functions and some
  related numbers.
\newblock {\em Acta Arith.}, 33(2):111--149, 1977.

\bibitem{kc}
K.~Chandrasekharan.
\newblock {\em Elliptic Functions}.
\newblock Springer, 1980.

\bibitem{copson}
E.~T. Copson.
\newblock {\em An introduction to the theory of functions of a complex
  variable}.
\newblock Oxford University Press, 1935.

\bibitem{gab}
A.~Gabrielov.
\newblock Complements of subanalytic sets and existential formulas for analytic
  functions.
\newblock {\em Invent. Math.}, 125(1):1--12, 1996.

\bibitem{khovanski}
A.~G. Hovanski\u{\i}.
\newblock A class of systems of transcendental equations.
\newblock {\em Dokl. Akad. Nauk SSSR}, 255(4):804--807, 1980.

\bibitem{jks}
G.~Jones, J.~Kirby, and T.~Servi.
\newblock Local interdefinability of {W}eierstrass elliptic functions.
\newblock {\em J. Inst. Math. Jussieu}, 15(4):673--691, 2016.

\bibitem{jw}
G.~O. Jones and A.~J. Wilkie.
\newblock Locally polynomially bounded structures.
\newblock {\em Bull. Lond. Math. Soc.}, 40(2):239--248, 2008.

\bibitem{langalg}
S.~Lang.
\newblock {\em Algebra, 3rd edition}.
\newblock Springer, 2002.

\bibitem{macwp}
A.~Macintyre.
\newblock The elementary theory of elliptic functions 1: the formalism and a
  special case.
\newblock O-mininmal strutures: Lisbon 2003 Proceedings of a summer school,
  2005.

\bibitem{mcculloch2020nondefinability}
R.~McCulloch.
\newblock A nondefinability result for expansions of the ordered real field by
  the {W}eierstrass $\wp$ function, 2020.

\bibitem{ps-uniformdefforP}
Y.~Peterzil and S.~Starchenko.
\newblock Uniform definability of the {W}eierstrass {$\wp$} functions and
  generalized tori of dimension one.
\newblock {\em Selecta Math. (N.S.)}, 10(4):525--550, 2004.

\bibitem{newtonpetstar}
Y.~Peterzil and S.~Starchenko.
\newblock Complex analytic geometry in a nonstandard setting.
\newblock In {\em Model theory with applications to algebra and analysis.
  {V}ol. 1}, volume 349 of {\em London Math. Soc. Lecture Note Ser.}, pages
  117--165. Cambridge Univ. Press, Cambridge, 2008.

\bibitem{axforj}
J.~Pila and J.~Tsimerman.
\newblock Ax-{S}chanuel for the {$j$}-function.
\newblock {\em Duke Math. J.}, 165(13):2587--2605, 2016.

\bibitem{sharipov-sukhov}
R.~Sharipov and A.~Sukhov.
\newblock On {CR}-mappings between algebraic {C}auchy-{R}iemann manifolds and
  separate algebraicity for holomorphic functions.
\newblock {\em Trans. Amer. Math. Soc.}, 348(2):767--780, 1996.

\bibitem{silvermanadvanced}
J.~H. Silverman.
\newblock {\em Advanced topics in the arithmetic of elliptic curves}, volume
  151 of {\em Graduate Texts in Mathematics}.
\newblock Springer-Verlag, New York, 1994.

\bibitem{silverman}
J.~H. Silverman.
\newblock {\em Arithmetic of Elliptic curves, 2nd edition}.
\newblock Springer, 2016.

\bibitem{duval}
P.~D. Val.
\newblock {\em Elliptic functions and elliptic curves}, volume~9 of {\em LMS
  lecture notes}.
\newblock Cambridge University Press, 1973.

\bibitem{wilkieexp}
A.~J. Wilkie.
\newblock Model completeness results for expansions of the ordered field of
  real numbers by restricted {P}faffian functions and the exponential function.
\newblock {\em J. Amer. Math. Soc.}, 9(4):1051--1094, 1996.

\bibitem{123modularforms-zagier}
D.~Zagier.
\newblock Elliptic modular forms and their applications.
\newblock In {\em The 1-2-3 of modular forms}, Universitext, pages 1--103.
  Springer, Berlin, 2008.

\end{thebibliography}
\end{document}